\documentclass[a4paper,reqno]{amsart}
\usepackage{indentfirst}
\usepackage{amsfonts}
\usepackage{amscd}
\usepackage{amsthm}
\usepackage{amssymb}
\usepackage{amsmath}
\usepackage{thmtools}
\usepackage{epsfig}
\usepackage{graphicx}
\usepackage{slashed}
\usepackage[all,cmtip]{xy}
\usepackage{tikz}
\usepackage{todonotes}
\usetikzlibrary{decorations.markings}
\usepackage{enumitem}
\usepackage{setspace}
\usepackage{hyperref}

\usetikzlibrary{matrix,arrows,positioning}

\def\itemNum$#1${\item $\displaystyle#1$
   \hfill\refstepcounter{equation}(\theequation)}

\newtheorem{Lem}{Lemma}[section]
\newtheorem{Prop}[Lem]{Proposition}

\theoremstyle{plain}
\newtheorem{Thm}[Lem]{Theorem}

\newtheorem{Cor}[Lem]{Corollary}

\theoremstyle{definition}

\newcommand{\rad}{\text{\textnormal{rad}}}

\newcommand{\Spec}{\text{\textnormal{Spec}}}

\mathchardef\mhyphen="2D

\newcommand{\LL}{{\mathbb L}}

\begin{document}

\title{The motive of the classifying stack of the orthogonal group}

\author[A. Dhillon]{Ajneet Dhillon}
\address{Department of Mathematics\\
The University of Western Ontario\\
London, Ontario}
\email{adhill3@uwo.ca}

\author[M.\,B. Young]{Matthew B. Young}
\address{Department of Mathematics\\
The University of Hong Kong\\
Pokfulam, Hong Kong}
\email{myoung@maths.hku.hk}

\date{\today}


\begin{abstract}
We compute the motive of the classifying stack of an orthogonal group in the Grothendieck ring of stacks over a field of characteristic different from two.
\end{abstract}

\maketitle

\setcounter{footnote}{0}

\section{Introduction}

The Grothendieck ring of stacks over a field $k$ has been introduced by a number of authors \cite{behrend2007}, \cite{ekedahl2009b}, \cite{joyce2007b}, \cite{toen2005}. Denote this ring by $\hat{K}_0(Var_k)$. An algebraic group $G$ defined over $k$ is called special if any $G$-torsor over a $k$-variety is locally trivial in the Zariski topology. General linear, special linear and symplectic groups are special. Special orthogonal groups are not special in dimensions greater than two. Serre proved that special groups are linear and connected \cite{serre1958}. Over algebraically closed fields, the special groups were classified by Grothendieck \cite{grothendieck1958}.

For a special group $G$, the motive $[G]$ is invertible in $\hat{K_0}(Var_k)$ and its inverse is equal to the motive of the classifying stack $BG$. This naturally raises the problem of computing the motive of $BG$ when the group $G$ is not special. For finite group schemes a number of examples were computed in \cite{ekedahl2009}. The case of groups of positive dimension is more difficult. In \cite{bergh2014} it was shown that $[BPGL_n ] =[PGL_n]^{-1}$ for $n=2$ or $3$ with mild restrictions on the field $k$.

The main result of this paper, Theorem \ref{thm:motiveBOn}, computes the motive of the classifying stack of an orthogonal group over a field whose characteristic is not two. In odd dimensions the result is that the motive is equal to the inverse of the motive of the split special orthogonal group in the same dimension. To prove Theorem \ref{thm:motiveBOn} we first compute the motive of the variety of nondegenerate quadratic forms of fixed dimension. This motive was already computed in \cite{belkale2003}, using results of \cite{macwilliams1969}. Our computation is different, relying on generating function techniques. Using Theorem \ref{thm:motiveBOn} we are able to compute the motives of classifying stacks of the special orthogonal groups in odd dimensions.

\subsection*{Notation} We will work over a base field $k$ with ${\rm char}(k)\ne 2$.
If $n$ is a non-negative integer we denote by $[n]_{\LL}$ the $n$th Gaussian polynomial in the Lefschetz motive $\LL$. Explicitly,
\[
[n]_\LL = 1 + \LL + \cdots + \LL^{n-1}.
\]
The Gaussian polynomials $[n]{_\LL}!$ and ${n \brack r}_\LL$
are defined in the usual way. The class of the Grassmannian $Gr(r,n)$ in the ring
$\hat{K}_0(Var_k)$ is then ${n \brack r}_\LL$.

\subsubsection*{Acknowledgements}
The authors thank the Institute for Mathematical Sciences at the National University of Singapore for support and hospitality during the program `The Geometry, Topology and Physics of Moduli Spaces of Higgs Bundles', where this work was initiated.

\section{Preliminaries}
\subsection{The Grothendieck ring of stacks}
Fix a ground field $k$. Let $K_0(Var_k)$ be the Grothendieck ring of varieties over $k$. Its underlying abelian group is generated by symbols $[X]$, with $X$ a $k$-variety, modulo the relations $[X] = [Y]$ if $X$ and $Y$ are isomorphic and
\[
[X]= [X \backslash Z ] + [Z]
\]
if $Z \subset X$ is a closed subvariety. Cartesian product of varieties gives $K_0(Var_k)$ the structure of a commutative ring with identity $1 = [\Spec \, k]$. The Lefschetz motive is defined to be $\LL = [ \mathbb{A}_k^1]$.

The Grothendieck ring of stacks, $\hat{K}_0(Var_k)$, is the dimensional completion of $K_0(Var_k)$ defined as follows \cite{behrend2007}. Let $F^m \subset K_0(Var_k)[\LL^{-1}]$ be the additive subgroup generated by those $\LL^{-d}[X]$ with $\dim X - d \leq -m $. This defines a descending filtration of $K_0(Var_k)[\LL^{-1}]$ and $\hat{K}_0(Var_k)$ is the completion with respect to this filtration.

In this paper all stacks are assumed to be Artin stacks that are locally of finite type, all of whose geometric stabilizers are linear algebraic groups. Following \cite{behrend2007}, a stack $\mathfrak{X}$ is called essentially of finite type if it admits a stratification $\mathfrak{X} =\cup_{i=1}^{\infty} \mathfrak{X}_i$ by finite type, locally closed substacks with $\lim_{i \rightarrow \infty} \dim \mathfrak{X}_i =- \infty$. Any stack that is essentially of finite type admits a stratification of the above type with $\mathfrak{X}_i$ a global quotient stack of a variety $X_i$ by a general linear group $GL_{n_i}$. Given such a stratification, put
\[
[\mathfrak{X}] = \sum_{i=1}^{\infty} \frac{[X_i]}{[GL_{n_i}]}.
\]
This defines a motivic class $[\mathfrak{X}] \in \hat{K}_0(Var_k)$ that is independent of the choice of stratification of $\mathfrak{X}$ \cite[Lemma 2.3]{behrend2007}.

\begin{Lem}[{\cite[Lemma 2.5]{behrend2007}}]
\label{lem:stackTorsor}
Let $\mathfrak{X}$ be a stack which is essentially of finite type and let $P \rightarrow \mathfrak{X}$ be a torsor for a linear algebraic group $G$. Then $P$ is essentially of finite type. Moreover, if $G$ is special, then $[P] = [\mathfrak{X}] [G]$ in $\hat{K}_0(Var_k)$.
\end{Lem}

In particular, if $G$ is special, applying Lemma \ref{lem:stackTorsor} to the universal $G$-torsor $\Spec \, k  \rightarrow BG$ shows that $[BG]=[G]^{-1}$. This equality is called the universal $G$-torsor relation.

More generally, if $X$ is a variety acted on by a linear algebraic group $G$, then the quotient stack $X \slash G$ has a class in $\hat{K}_0(Var_k)$. For any closed embedding $G \hookrightarrow GL_N$ there is an isomorphism of stacks $X \slash G \simeq (X \times_G GL_N) \slash GL_N$. Since $GL_N$ is special, Lemma \ref{lem:stackTorsor} implies that
\begin{equation}
\label{eq:stackMotive}
[X \slash G] =  \frac{[X \times_G GL_N]}{[GL_N]}
\end{equation}
in $\hat{K}_0(Var_k)$.

\subsection{Orthogonal groups}
Assume that the ground field $k$ is not of characteristic two. Let $V$ be a finite dimensional vector space over $k$ and let $Q: V \rightarrow k$ be a quadratic form. The radical of $Q$ is the subspace of $V$ defined by
\[
\rad_Q = \{ v \in V \; \vert \; Q(v+w) =Q(v) + Q(w) \;\;\; \forall w \in V \}.
\]
The rank of $Q$ is $\dim V - \dim \rad_Q$. The quadratic form $Q$ is called nondegenerate if $\rad_Q = \{0 \}$.

Given a nondegenerate quadratic form $Q$, denote by $O(Q)$ its group of isometries. If the field $k$ is algebraically closed, then there is a unique nondegenerate quadratic form on $k^n$ up to equivalence. The corresponding orthogonal group is unique up to isomorphism. If $k$ is not algebraically closed, then there will in general exist inequivalent nondegenerate quadratic forms on $k^n$, leading to different forms of orthogonal groups.

For each $n \geq 1$, there is a canonical nondegenerate split quadratic form on $k^n$. Explicitly,
\[
Q_{2r} = x_1 x_2 + \cdots + x_{2r-1} x_{2r}
\]
and
\[
Q_{2r+1} = x_0^2 + x_1 x_2 + \cdots + x_{2r-1} x_{2r}.
\]
Define $O_n=O(Q_n)$ and $SO_n=SO(Q_n)$.

\section{The motive of $BO(Q)$}

\subsection{Filtration of the space of quadratic forms}
Recall that ${\rm char}(k) \neq 2$.  Denote by $Quad_n \simeq \mathbb{A}_k^{n + 1 \choose 2}$ the affine space of quadratic forms on $k^n$. The group $GL_n$ acts on $Quad_n$ by change of basis. For each $0 \leq r \leq n$, let $Quad_{n, \leq r} \subset Quad_n$ denote the closed subvariety of quadratic forms whose rank is at most $r$. This gives an increasing filtration of $Quad_n$ by closed subvarieties. Interpreted in $K_0(Var_k)$, this implies the identity
\begin{equation}
\label{eq:totalDecomp}
\LL^{n + 1 \choose 2} = \sum_{r=0}^n [Quad_{n,r}]
\end{equation}
with $Quad_{n,r}$ the subvariety of quadratic forms of rank $r$. Denote by $Gr(m,n)$ the Grassmannian of $m$-planes in $k^n$.

\begin{Prop}
\label{prop:zarTriv}
For each $0 \leq r \leq n$, the map
\[
\pi: Quad_{n,r}  \rightarrow Gr(n-r,n), \;\;\; Q \mapsto \rad_Q
\]
is a Zariski locally trivial fibration with fibres isomorphic to $Quad_{r,r}$.
\end{Prop}

\begin{proof}
Identify $Gr(n-r,n)$ with the quotient of the variety of $(n-r)\times n$ matrices of rank $n-r$ by the left action of $GL_{n-r}$. Fix coordinates $x_1, \dots, x_n$ on $k^n$. Consider the $(n-r)$-plane $k^{n-r} \subset k^n$ with coordinates $x_1, \dots, x_{n-r}$. A Zariski open set $U \subset Gr(n-r,n)$ containing $k^{n-r}$ is given by the $(n-r) \times n$ matrices of the form
\[
\left( \begin{matrix} \mathbf{1}_{n-r} & B  \end{matrix} \right)
\]
with $\mathbf{1}_{n-r}$ the $(n-r)\times (n-r)$ identity matrix and $B$ an arbitrary $(n-r) \times r$ matrix. The plane $k^{n-r}$ corresponds to the matrix $B=0$. Note that
\[
\left( \begin{matrix} \mathbf{1}_{n-r} & B  \end{matrix} \right) = \left( \begin{matrix} \mathbf{1}_{n-r} & 0   \end{matrix} \right) \cdot \left( \begin{matrix} \mathbf{1}_{n-r} & B \\ 0 & \mathbf{1}_r \end{matrix} \right).
\]
Let $g_B = \left( \begin{matrix} \mathbf{1}_{n-r} & B \\ 0 & \mathbf{1}_r \end{matrix} \right) \in GL_n$, viewed as an automorphism of $k^n$.

Suppose that $Q \in \pi^{-1}(U)$. Then there exists a unique matrix $B(Q)$ such that $\rad_Q = g_{B(Q)} (k^{n-r}) \subset k^n$. The quadratic form $g_{B(Q)} \cdot Q$ is the pullback of a nondegenerate quadratic form $\varphi_Q$ in the variables $x_{n-r+1}, \dots, x_n$. A trivialization of $\pi$ over $U$ is then given by
\[
\pi^{-1}(U) \rightarrow U \times Quad_{r,r}, \;\;\;\;\;\; Q \mapsto (\rad_Q, \varphi_Q).
\]
This argument can be repeated, replacing $k^{n-r}$ with the $(n-r)$-plane with coordinates labelled by a $(n-r)$-element subset $I \subset \{1, \dots, n\}$. This gives a Zariski open cover of $Gr(n-r,n)$ over which $\pi$ trivializes. 
\end{proof}

\begin{Cor}
\label{cor:recurrence}
The identity
\[
\LL^{n + 1 \choose 2} = \sum_{r=0}^n {n \brack n-r}_{\LL} [Quad_{r,r}]_{\LL}
\]
holds in the ring $K_0(Var_k)$.
\end{Cor}

\begin{proof}
It follows from Proposition \ref{prop:zarTriv} that $[Quad_{n,r}] = [Gr(n-r,n)] [Quad_{r,r}]$. Since $[Gr(n-r,n)]={n \brack n-r}_\LL$, the desired identity is implied by equation \eqref{eq:totalDecomp}.
\end{proof}

\subsection{Solving the recurrence}

In this section we will solve the recurrence relation for $[Quad_{n,n}]$ given in Corollary \ref{cor:recurrence}. In fact, the motives $[Quad_{n,r}]$ were already computed in \cite[Theorem 13.5]{belkale2003}, where it was shown that $[Quad_{n,r}]$ satisfies a certain three step recurrence relation with coefficients in $\mathbb{Z}[\LL]$. This recurrence relation, with $\mathbb{L}$ replaced by $q$, was previously solved in \cite{macwilliams1969} to find the number of $\mathbb{F}_q$-rational points of $Quad_{n,r}$. Hence $[Quad_{n,r}]$ is given by the same formula, with $q$ replaced with $\LL$. We present here an alternative computation of $[Quad_{n,n}]$, and therefore also $[Quad_{n,r}]$ by Proposition \ref{prop:zarTriv}, using generating functions.

We form the exponential generating function for the motives $[Quad_{n,n}]$,
\[
G(x)=\sum_{n\ge 0} \frac{[Quad_{n,n}] x^n}{[n]_{\LL}!}.
\]
Consider also the auxiliary generating functions
\[
P_{\rm even}(x) = \sum_{k\ge 0}\frac{x^{2k}}{[2k]_{\LL}!}
\prod_{i=1}^k (\LL^{2k+1}-\LL^{2i})
\]
and 
\[
P_{\rm odd}(x)  = \sum_{k\ge 0}\frac{x^{2k+1}}{[2k+1]_{\LL}!}
\prod_{i=0}^k (\LL^{2k+1}-\LL^{2i}).
\]
We will show that
\[
G(x) = P_{\rm even}(x) + P_{\rm odd}(x),
\]
thereby solving the recurrence relation.

\begin{Prop}
\label{prop:lExpProp}
Denote by $\exp_\LL(x)$ the $\LL$-deformed exponential series:
\[
\exp_\LL(x) = \sum_{n\ge 0} \frac{x^n}{[n]_{\LL}!}.
\]
The following equality holds:
\[
G(x) = \frac{\prod_{i\ge 1}(1+(1-\LL)x\LL^i)}{\exp_\LL(x)}.
\]
\end{Prop}
\begin{proof}
To ease notation set $\mathcal{Q}_n = [Quad_{n,n}]$. Using Corollary \ref{cor:recurrence} we find that
\begin{align*}
G(x) &= \sum_{n\ge 0} \frac{\mathcal{Q}_n}{[n]_{\LL}!}x^n \\
&= \sum_{n\ge 0} \left( \LL^{ n+1 \choose 2} -  \sum_{r=0}^{n-1} {n \brack n-r}_\LL \mathcal{Q}_r \right)\frac{x^n}{[n]_{\LL}!} \\
&= \sum_{n\ge 0}  \left( \LL^{n+1 \choose 2} -  
\sum_{r=0}^{n-1} \frac{[n]_{\LL}!}{[n-r]_{\LL}! [r]_{\LL}!}\mathcal{Q}_r \right)\frac{x^n}{[n]_{\LL}!} \\
&= \sum_{n\ge 0}  \left( \LL^{n+1 \choose 2}\frac{x^{n}}{[n]_{\LL}!} -  
\sum_{r=0}^{n-1} \frac{\mathcal{Q}_rx^r}{[r]_{\LL}!} \frac{x^{n-r}}{[n-r]_{\LL}!} \right) \\
&= \sum_{n\ge 0}  \LL^{n+1 \choose 2}\frac{x^{n}}{[n]_{\LL}!}  
-\sum_{n\ge 0}\sum_{r=0}^n \frac{\mathcal{Q}_r x^{n-r}}{[r]_{\LL}! [n-r]_{\LL}!}
+\sum_{n\ge 0}\frac{\mathcal{Q}_n x^n}{[n]_{\LL}!}\\
&= \sum_{n\ge 0}  \LL^{n+1 \choose 2}\frac{x^{n}}{[n]_{\LL}!}
-\exp_\LL(x)G(x) + G(x).
\end{align*}
Hence
\[
G(x) = \frac{\sum_{n\ge 0}  \LL^{n+1 \choose 2}\frac{x^{n}}{[n]_{\LL}!}}{\exp_\LL(x)}.
\]
Since
\[
[n]_{\LL}! = \frac{(1-\LL)(1-\LL^2)\cdots (1-\LL^n)}{(1-\LL)^n}
\]
we have
\begin{eqnarray*}
\sum_{n\ge 0}  \LL^{n+1 \choose 2}\frac{x^{n}}{[n]_{\LL}!} &=& 
\sum_{n\ge 0}  \frac{\LL^{n+1 \choose 2}(1-\LL)^nx^{n}}
{(1-\LL)(1-\LL^2)\cdots (1-\LL^n)} \\
&=& \prod_{i\ge 1}(1+(1-\LL)x\LL^i),
\end{eqnarray*}
where the second second equality follows from \cite[Proposition 1.8.6]{stanley2012}. This completes the proof.
\end{proof}

It will be convenient to make the change of variables $g(x) = G(\frac{x}{1 - \LL})$.

\begin{Prop}
\label{prop:closed}
We have
\begin{align*}
g(x) &= (1-x)\prod_{i\ge 1} (1-x^2\LL^{2i}) \\
     &= (1-x)\sum_{k\ge 0} \frac{(-1)^k x^{2k}\LL^{k(k+1)}}
     {(1-\LL^2)(1-\LL^4)\cdots (1-\LL^{2k})}.
\end{align*}
\end{Prop}

\begin{proof}
We compute
\begin{align*}
\exp_\LL(x) &= \sum_{n\ge 0} \frac{x^n}{[n]_{\LL}!} \\
&= \sum_{n\ge 0} \frac{x^n(1-\LL)^n}{(1-\LL)(1-\LL^2)\cdots (1-\LL^n)} \\
&= \frac{1}{\prod_{i\ge 0} (1-(1-\LL)x\LL^i)},
\end{align*}
where the last equality is via \cite[page 74]{stanley2012}.
The first assertion now follows from Proposition \ref{prop:lExpProp}. The second follows from the first by \cite[Proposition 1.8.6]{stanley2012}.
\end{proof}

Similarly, make the change of variables $p_{\rm even}(x)=P_{\rm even}(\frac{x}{1-\LL})$ and $p_{\rm odd}(x)=P_{\rm odd}(\frac{x}{1-\LL})$.

\begin{Prop}
\label{prop:pIden}
We have
\[
p_{\rm even}(x) = \sum_{k\ge 0}\frac{(-1)^k x^{2k}\LL^{k(k+1)}}
{(1-\LL^2)(1-\LL^4)\cdots (1-\LL^{2k})}
\]
and
\[
p_{\rm odd}(x) = \sum_{k\ge 0}\frac{(-1)^{k+1}x^{2k}\LL^{k(k+1)}}
{(1-\LL^2)(1-\LL^4)\cdots(1-\LL^{2k})}.
\]
\end{Prop}

\begin{proof}
The generating function $P_{\rm even}$ can be rewritten as
\[
P_{\rm even}(x) = \sum_{k\ge 0}\frac{(1-\LL)^{2k}x^{2k}}
{(1-\LL)(1-\LL^2)\cdots (1-\LL^{2k})}
\prod_{i=1}^k (\LL^{2k+1}-\LL^{2i}).
\]
Then we have
\begin{align*}
p_{\rm even}(x) &= \sum_{k\ge 0}\frac{x^{2k}}
{(1-\LL)(1-\LL^2)\cdots (1-\LL^{2k})}
\prod_{i=1}^k (\LL^{2k+1}-\LL^{2i}) \\
&=  \sum_{k\ge 0}\frac{x^{2k}\LL^{k(k+1)}}
{(1-\LL)(1-\LL^2)\cdots (1-\LL^{2k})}
\prod_{i=1}^k (\LL^{2(k-i)+1}-1) \\
&= \sum_{k\ge 0}\frac{(-1)^k x^{2k}\LL^{k(k+1)}}
{(1-\LL^2)(1-\LL^4)\cdots (1-\LL^{2k})}.
 \\
\end{align*}
The calculation for $p_{\rm odd}$ is similar.
\end{proof}

\begin{Cor}
\label{cor:recurSoln}
The following identity holds in $\hat{K}_0(Var_k)$ :
\[
G(x) = P_{\rm even}(x) + P_{\rm odd}(x).
\]
\end{Cor}

\begin{proof}
As $(1-\LL)$ is a unit in $\hat{K}_0(Var_k)$ it suffices to show that
\[ g(x) = p_{\rm even}(x) + p_{\rm odd}(x).\]
This follows from Propositions \ref{prop:closed} and \ref{prop:pIden}.\end{proof}

\subsection{The main theorem}

We now state the main result. 

\begin{Thm}
\label{thm:motiveBOn}
Let $k$ be a field whose characteristic is not 2 and let $n \geq 1$. For any nondegenerate quadratic form $Q$ on $k^n$, the following equality holds in $\hat{K}_0(Var_k)$:
\[
[BO(Q)] = \left\{  \begin{array}{cl} \displaystyle \LL^{-r} \prod_{i=0}^{r-1} (\LL^{2r} - \LL^{2i})^{-1}, & \mbox{ if } n=2r+1 \\  \displaystyle \LL^{r} \prod_{i=0}^{r-1} (\LL^{2r} - \LL^{2i})^{-1} , & \mbox{ if } n=2r. \end{array} \right.
\]
\end{Thm}

\begin{proof}
The  subvariety $Quad_{n,n} \subset Quad_n$ is stable under the action of $GL_n$ on $Quad_n$. Pick $Q \in Quad_{n,n}$. This gives rise to an orbit morphism $GL_n \rightarrow Quad_{n,n}$. Since $\pi: GL_n \rightarrow GL_n \slash O(Q)$ is a uniform categorical quotient \cite[Theorem 1.1]{mumford1994}, the orbit morphism factors through a unique morphism $\psi: GL_n \slash O(Q) \rightarrow Quad_{n,n}$. We claim that $\psi$ is an isomorphism.

Let $\overline{k}$ be an algebraic closure of $k$. Base change gives a morphism
\[
\overline{\pi}: GL_{n,\overline{k}} \rightarrow  GL_n \slash O(Q) \times_k \overline{k}
\]
which is a categorical quotient for the action of $O(Q)_{\overline{k}}$ on $GL_{n,\overline{k}}$. Here $GL_{n,\overline{k}}$ denotes the general linear group over $\overline{k}$ while $O(Q)_{\overline{k}}$ denotes orthogonal group of the quadratic form $Q\times_k \overline{k}$ on $\overline{k}^n$. The universal property of categorical quotients implies
\[
GL_n \slash O(Q) \times_k \overline{k} \simeq GL_{n,\overline{k}} \slash O(Q)_{\overline{k}}.
\]
Using this isomorphism and applying base change to $\psi$ gives
\[
\overline{\psi}: GL_{n,\overline{k}} \slash O(Q)_{\overline{k}} \rightarrow Quad_{n,n} \times_k \overline{k}.
\]
Since $Quad_{n,n} \times_k \overline{k}$ is homogeneous under the action of $GL_{n, \overline{k}}$ with stabilizer $O(Q)_{\overline{k}}$, the map $\overline{\psi}$ is an isomorphism. By faithfully flat descent it follows that $\psi$ itself is an isomorphism.

Identifying $BO(Q)$ with the quotient stack $\Spec \, k \slash O(Q)$, equation \eqref{eq:stackMotive} gives
\[
[BO(Q)] = \left[ \frac{GL_n \slash O(Q)}{GL_n} \right] = \frac{[GL_n \slash O(Q)]}{[GL_n]} =  \frac{[Quad_{n,n}]}{[GL_n]}.
\]
Using Corollary \ref{cor:recurSoln}, we read off from $P_{even}$ and $P_{odd}$ the equality
\[
[Quad_{n,n}] = \left\{  \begin{array}{cl} \displaystyle \prod_{i=0}^r (\LL^{2r+1}-\LL^{2i}), & \mbox{ if } n=2r+1 \\  \displaystyle \prod_{i=1}^r (\LL^{2r+1}-\LL^{2i}) , & \mbox{ if } n=2r. \end{array} \right.
\]
If $n=2r+1$, we have
\begin{eqnarray*}
\frac{[Quad_{2r+1,2r+1}]}{[GL_{2r+1}]} &=& \frac{\prod_{i=0}^r (\LL^{2r+1}-\LL^{2i})}{\prod_{i=0}^{2r} (\LL^{2r+1}-\LL^i)} \\
&=& \prod_{i=0}^{r-1} (\LL^{2r+1}-\LL^{2i+1})^{-1}\\
&=& \LL^{-r} \prod_{i=0}^{r-1} (\LL^{2r}-\LL^{2i})^{-1}
\end{eqnarray*}
which is the desired result. The calculation for $n$ even is analogous.
\end{proof}

\begin{Cor}
\label{cor:motiveBSOodd}
Suppose that $n \geq 3$ is odd and let $Q$ be a nondegenerate quadratic form on $k^n$. Then $[BO(Q)]=[SO_n]^{-1}$. Moreover, $[BSO(Q)]= [SO_n]^{-1}$.
\end{Cor}
\begin{proof}
Since $n \geq 3$, the split group $SO_n$ is semisimple. According to \cite[Lemma 2.1]{behrend2007},
\[
[SO_{2r+1}] = \LL^r \prod_{i=0}^{r-1} ( \LL^{2r} - \LL^{2i}).
\]
Comparing this expression with Theorem \ref{thm:motiveBOn} gives the first statement. Continuing, if $Q$ is a nondegenerate quadratic form in odd dimensions there is an isomorphism $O(Q) \simeq \mu_2 \times SO(Q)$. It is shown in \cite[Proposition 3.2]{ekedahl2009} that $[B \mu_2 ] =1$. Hence
\[
[BO(Q)] = [B\mu_2 \times B SO(Q)] =  [B\mu_2] [B SO(Q)] =  [B SO(Q)].
\]
The second statement now follows from the first.
\end{proof}

Since $PGL_2 \simeq SO_3$ over any field, Corollary \ref{cor:motiveBSOodd} recovers the first part of \cite[Theorem A]{bergh2014} as a special case.

It follows from Corollary \ref{cor:motiveBSOodd} that the universal torsor relations are satisfied for split special orthogonal groups in odd dimensions. In particular, the universal $SO_{2n+1}(\mathbb{C})$-torsor relation holds. In \cite[Theorem 2.2]{ekedahl2009c} it is shown that for any non-special connected reductive complex algebraic group $G$ there exists a $G$-torsor $P \rightarrow X$ over a variety such that $[P]$ is not equal to $[G] [X]$. Therefore, the universal $G$-torsor relation does not imply the general $G$-torsor relation, answering a question posed in \cite[Remark 3.3]{behrend2007}. In the recent paper \cite{bergh2014} the groups $PGL_2(\mathbb{C})$ and $PGL_3(\mathbb{C})$ were also shown to answer this question.

\bibliographystyle{plain}
\bibliography{mybib}

\end{document}